\RequirePackage{fix-cm} 
\documentclass[a4paper,twoside,12pt,reqno]{amsart}

\usepackage{fixltx2e}     

\usepackage[english]{babel}
\usepackage[latin1]{inputenc}

\usepackage{indentfirst,verbatim}

\usepackage{amsmath,amsfonts,amssymb,amsgen,amsbsy,eucal,mathrsfs,dsfont}
\usepackage{stmaryrd}

\usepackage[square,numbers]{natbib}
\usepackage{a4wide,verbatim}

\usepackage{epsfig,rotating,color}
\usepackage{multicol}
\usepackage{tikz,pgfplots}
\usepackage{graphicx}
\usepackage{mathtools}
\usepackage{textcomp}

\usepackage{caption}

\usepackage{hyperref}

\usepackage[square,numbers]{natbib}
\usepackage{systeme}

\hypersetup{
colorlinks=true, 
breaklinks=true, 
urlcolor= blue, 
linkcolor= blue, 
bookmarksopen=true, 
pdftitle={}, 
pdfauthor={}, 
pdfsubject={}, 
citecolor=red,  
}

\newcommand{\R}{\mathbb{R}}

\newcommand{\ud}{\mathrm{d}}

\newcommand{\half}{{\textstyle{1\over2}}}

\usepackage{amsthm}
\newtheorem{thm}{Theorem}[section]
\newtheorem{definition}[thm]{Definition}
\newtheorem{lem}[thm]{Lemma}
\newtheorem{pro}[thm]{Proposition}

\theoremstyle{remark}
\newtheorem*{rem}{Remark}

\usepackage{enumitem}
\newlist{steps}{enumerate}{1}
\setlist[steps, 1]{label = Step \arabic*:}

\usepackage[textsize=tiny,textwidth=27mm]{todonotes}

\newcommand*{\bydef}{\overset{\rm def}{=}}
\newcommand*{\norm}[1]{\left\Vert #1\right\Vert}

\newcommand{\eqdef}{\stackrel{\text{\tiny{def}}}{=}}

\title[Regularized scalar conservation law]{\bf Convergence rate for a regularized scalar conservation law}

\author[Guelmame and Houamed]{Billel Guelmame and Haroune Houamed}

\newcommand{\nfont}{\fontshape{n}\selectfont}

\address{({\nfont\textbf{Billel Guelmame}})  UMPA, CNRS,  ENS Lyon, Universit\'e de Lyon, France.} 
\email{billel.guelmame@ens-lyon.fr}

\address{({\nfont\textbf{Haroune Houamed}}) New York University Abu Dhabi, Abu Dhabi, United Arab Emirates.} 
\email{haroune.houamed@nyu.edu}

\let\oldtocsection=\tocsection
 
\let\oldtocsubsection=\tocsubsection

\renewcommand{\tocsection}[2]{\hspace{0em}\oldtocsection {#1}{#2}}
\renewcommand{\tocsubsection}[2]{\hspace{2em}\oldtocsubsection{#1}{#2}}

\numberwithin{equation}{section}

\usepackage{tabularx}
\usepackage[pagewise]{lineno} 

\begin{document}
\maketitle

\begin{center}
    \itshape \small  In the memory of Ahmed Blidia
\end{center}

\begin{abstract} 
	This work revisits a recent finding by the first author concerning the local convergence of a regularized scalar conservation law. We significantly improve the original statement by establishing a global convergence result within the Lebesgue spaces $L^\infty_{\mathrm{loc}}(\mathbb{R}^+;L^p(\mathbb{R}))$, for any $p \in [1,\infty)$, as the regularization parameter $\ell$ approaches zero. Notably, we demonstrate that this stability result is accompanied by a quantifiable rate of convergence. A key insight in our proof lies in the observation that the fluctuations of the solutions remain under control in low regularity spaces, allowing for a potential quantification of their behavior in the limit as $\ell\to 0$. This is achieved through a careful asymptotic analysis of the perturbative terms in the regularized equation, which, in our view, constitutes a pivotal contribution to the core findings of this paper.
\end{abstract}
 
\medskip

 {\small{\bf AMS Classification :} 35L65; 35L67; 35Q35.

\medskip

{\bf Key words :} Scalar conservation laws; Regularization; Ole\u{\i}nik inequality;  Convergence rate. }

\tableofcontents

\section{Introduction}

\subsection{Motivation}
The occurrence of shock formation in solutions of the scalar conservation laws
\begin{align}\label{SCL} \tag{Scl}
\partial_t u + \partial_x  f(u) = 0,
\end{align}
is a well-known phenomenon. Given any smooth initial data $u_0$, a unique strong solution exists. However, due to the nonlinear nature of the flux   $f$, discontinuous shock waves may develop in finite time. 
This behavior represents one of the challenges associated with nonlinear conservation laws.
In order to avoid the occurrence of shocks, various regularization techniques have been proposed. These regularizations aim to smooth out discontinuities by adding ``small'' terms to the equation, such  as diffusion and/or dispersion.
While diffusive regularizations are widely used, they tend to dissipate energy everywhere. On the other hand, the entropy solutions of \eqref{SCL} concentrate the energy dissipation at singularities.
Nevertheless, diffusive regularizations are considered as solid tools in establishing the existence of solutions and in justifying the a priori estimates via the vanishing viscosity method.
Dispersion regularizations lead to the appearance of spurious oscillations and fail, in general, to converge to the entropy solutions of \eqref{SCL}.

\subsection{The equations of our interest}
In order to introduce a regularization while preserving essential properties of the original equations, Clamond and Dutykh \cite{CD18} derived a non-diffusive, non-dispersive regularized Saint-Venant (rSV) system.
The study of traveling-wave solutions to the rSV system has been done in \cite{PPDC18}. 
Furthermore, the local well-posedness of that system and a construction of initial data leading to the appearance of singularities have been studied in \cite{LPP19}.
The rSV system has been generalized lately to regularize the barotropic Euler system \cite{GCJ22}.
Inspired by \cite{CD18}, and due to the complexity of studying the singular limit for those systems, Guelmame et al. \cite{GJCP22} proposed and studied the scalar non-diffusive, non-dispersive regularized Burgers  equation  
\begin{equation}\label{rB0} \tag{rB} 
\partial_t u^\ell + u^\ell \partial_x u^\ell = \ell^2  \left(   \partial_{txx}^3 u^\ell + 2 \partial_x u^\ell  \partial_{xx}^2 u^\ell  + u^\ell \partial_{xxx}^3 u^\ell \right),
\end{equation} 
where $\ell$ is a positive parameter.
The  equation \eqref{rB0} is Galilean invariant, it has been derived using a variational principle and it enjoys both Lagrangian and Hamiltonian structures. 
Smooth solutions to \eqref{rB0} conserve an $H^1$ energy, which prevents the appearing of discontinuous shocks, thanks to the Sobolev embedding $H^1(\R) \hookrightarrow C_b^0(\R)$.
In \cite{GJCP22}, the authors studied weakly singular shocks and cusped traveling-wave weak solutions of \eqref{rB0}. Additionally, they demonstrated that for every simple shock-wave entropy solution of the inviscid Burgers equation, there exists a corresponding monotonic traveling-wave dissipative solution of \eqref{rB0}. Notably, these solutions exhibit identical shock speed and energy dissipation rate as the original shock-wave solutions of the Burgers equation, which are recovered taking $\ell \to 0$.

In order to obtain general solutions to \eqref{rB0}, inspired by \cite{BC07a,BC07b}, the authors of \cite{GJCP22} proved the existence of two types of global weak solutions to \eqref{rB0}, conserving or dissipating the energy. The method of proof consists in utilizing two equivalent semi-linear system of ODEs, formulated in the Lagrangian coordinates. One system provides conservative solutions while the other yields to dissipative ones.
Conservative solutions maintain a constant energy for almost all time, including at singularities. They also fail to satisfy a one-sided Ole\u{\i}nik inequality, making them less accurate for regularizing entropy solutions of the Burgers equation.
Conversely, dissipative solutions concentrate the loss of the energy on the singularities and satisfy  the one-sided Ole\u{\i}nik inequality $\partial_x u^\ell(t,x) \leqslant 2/t$ for almost all $(t,x) \in (0,\infty) \times \R$.
The compactness of the dissipative solutions of \eqref{rB0} have been studied in \cite{GJCP22}. However, the equation satisfied in the limit was not identified at that time. 

In a recent work \cite{G23}, the first author considered the regularized scalar conservation law
\begin{equation}\label{rSCL} 
\partial_t u^\ell + \partial_x f(u^\ell) = \ell^2  \left(  \partial_{txx}^3 u^\ell + 2  f''(u^\ell) \partial_x u^\ell \partial_{xx}^2 u^\ell +f'(u^\ell) \partial_{xxx}^3  u^\ell  + \half   f'''(u^\ell)  (\partial_x u^\ell)^3 \right), 
\end{equation}
where $f$ is a uniformly convex flux. Notice that the regularized Burgers equation \eqref{rB0} is recovered taking $f(u) = u^2 /2$.
Using an approximation of \eqref{rSCL} involving a cut-off function, it has been proved in \cite{G23} that global weak dissipative solutions to \eqref{rSCL} exist. Moreover, as $\ell$ approaches zero, it has been shown that
\begin{equation}\label{compactness}
u^\ell \to u \quad \mathrm{in}\ L^\infty_{\text{loc}}(\R^+;L^p_{\text{loc}}(\R)),
\end{equation}
for any $p \in [1, \infty)$, where $u$ is the unique entropy solution of the scalar conservation law \eqref{SCL}. This gives a justification of the denomination ``regularization'' of the equation \eqref{rSCL}. The limit \eqref{compactness} was obtained via abstract compactness arguments which is why it was only established on compact sets without a determination of a convergence rate.
In this paper, we improve the latter result \eqref{compactness} by showing that it holds globally in space and  establishing an explicit convergence rate. More precisely, we will prove later on that
\begin{equation*}
	\norm {u^\ell -u}_{L^\infty([0,T];L^p(\mathbb{R}))} =O(\ell^{\frac {1}{2p}}),
\end{equation*}
for any $T>0$ and $p \in [1, \infty)$. We defer the discussion of this improvement to Section \ref{section:main}.

\subsection{Related equations} The rB equation \eqref{rB0} can be compared to the well-known dispersionless Camassa--Holm  equation \cite{CH93}
\begin{equation}\label{CH}
\partial_t u^\ell + 3 u^\ell \partial_x u^\ell =  \ell^2  \left(   \partial_{txx}^3 u^\ell + 2 \partial_x u^\ell  \partial_{xx}^2 u^\ell  + u^\ell \partial_{xxx}^3 u^\ell \right). \tag{CH}
\end{equation} 
The Camassa--Holm equation appears in modeling nonlinear wave propagation in the shallow-water regime.
Both  \eqref{rB0} and   \eqref{CH} conserve an $H^1$ energy for smooth solutions and they admit global weak conservative and dissipative solutions. Two key differences between the two equations are: 
(1) the   equation \eqref{CH}  is bi-Hamiltonian (therefore integrable) while only one Hamiltonian structure is known for the   equation \eqref{rB0}; 
(2) the    equation \eqref{rB0} is Galilean invariant while the equation \eqref{CH}   is not.  
The Galilean invariance is crucial from the physical point of view and also for proving mathematical results.
Indeed, due to the lack of the Galilean invariance, we could only prove that dissipative solutions of the   equation \eqref{CH} satisfy a one-sided Ole\u{\i}nik inequality involving a constant that blows-up as $\ell \to 0$.
This makes the singular limit $\ell \to 0$ for the   equation \eqref{CH} more challenging. 
To the authors' knowledge, this remains an open problem. 
However, in the presence of the viscosity in \eqref{CH}, and under a condition that $\ell$ is small compared to the viscosity parameter, the unique entropy solution of the equation $\partial_t u+ \partial_x (3u^2/2)=0$ is recovered by taking both parameters to zero \cite{CD16,CD17,CK06,H07}.

Another similar equation is the hyperelastic-rod wave equation \cite{D98a,D98b,DH00} 
\begin{equation}\label{Hyp-rod}
\partial_t u^\ell + 3 u^\ell \partial_x u^\ell = \ell^2  \left(  \partial_{txx}^3 u^\ell + \gamma \left( 2 \partial_x u^\ell  \partial_{xx}^2 u^\ell  + u^\ell \partial_{xxx}^3 u^\ell \right) \right),
\end{equation} 
where $\gamma$ is a real parameter.
The equation \eqref{Hyp-rod} describes radial deformation waves in cylindrical hyperelastic rods with a finite length and small amplitude.
Existence of global weak solutions to \eqref{Hyp-rod} has been established in \cite{CHK05,HR07}.
Observe that the Camassa--Holm equation is recovered taking $\gamma=1$ in \eqref{Hyp-rod}. 
It worths noting that the equation \eqref{Hyp-rod} satisfies a Galilean-like invariance property  only when $\gamma$ takes the values  $0$ or $3$. Setting $\gamma = 0$ yields to the Benjamin--Bona--Mahony equation \cite{BBM72}, which describes long surface gravity waves of small amplitude.
The value $\gamma = 3$, on the other hand, corresponds to the regularized equation \eqref{rSCL} with $f(u)=3u^2/2$ (or simply to \eqref{rB0} after a change of variables).
Therefore, we emphasize that the results established in \cite{G23} and in the present paper work for the hyperelastic-rod wave equation \eqref{Hyp-rod} with $\gamma=3$, as well.

The   equation \eqref{rB0} can also be compared to the Leray-type regularization proposed and studied by Bhat and Fetecau \cite{BF06,BF09}
\begin{equation}\label{BF}  
\partial_t u^\ell + u^\ell \partial_x u^\ell = \ell^2  \left(  \partial_{txx}^3 u^\ell  + u^\ell \partial_{xxx}^3 u^\ell \right),
\end{equation} 
which admits global solutions. Moreover, as $\ell \to 0$, solutions to \eqref{BF} converge, up to a subsequence, to a weak solution of the Burgers equation. 
Additionally, considering a simple Riemann problem with a decreasing initial data, the correct shock of the Burgers equation is recovered. 
However, for an increasing initial data, solutions of \eqref{BF} create  non-entropy jumps \cite{BF09}.

\subsection{Outline} 

This paper is organized as follows.  
In Section \ref{section:main}, we discuss some crucial basis of the equations of our interest, including a result about the existence of solutions to \eqref{rSCL}  and  we state the main theorem of this paper.  Then,  Section \ref{section:UB} is devoted to obtaining uniform bounds on the solutions to  viscous approximations of both equations \eqref{SCL} and \eqref{rSCL}. Thereafter, in Section \ref{section:decay}, we establish a decay estimate on the perturbative terms of \eqref{rSCL} (on its non-local form, see \eqref{rB} below) and prove the main result of the paper (Theorem \ref{main:thm}). At last, for clarity, we defer the recap on the definitions of the functional spaces utilized in this paper to Appendix \ref{appendix}, where, for the sake of completeness,  we also collect a few useful properties of these  spaces which apply in our proofs.

\section{Preliminaries and Main result}\label{section:main}

 Before we state our main result, allow us to prepare the ground around it by first setting up the essential assumptions in the paper and introducing the notion of solutions we are concerned with. Here, we are interested in the behavior, as $\ell$ tends to zero, of solutions to the regularized scalar conservation laws
 \begin{equation}\label{rB}
\partial_t u^\ell + \partial_x f(u^\ell)  + \ell^2  \partial_x P^\ell = 0, \qquad \quad  P^\ell - \ell^2  \partial_x^2 P^\ell = \half  f''(u^\ell) \left(\partial_x u^\ell\right)^2,
\end{equation}
  with an initial datum   $u_0\in H^1(\mathbb{R})$. Hereafter, we chose to lighten our notations by denoting $ \partial_{x}^2$ instead of $\partial_{xx}^2$.
 
 Henceforth,   the flux $f$ is assumed to be a regular uniformly convex function in the sense that \begin{gather}  \label{assum:flux}
f \in C^4(\R), \qquad 0 < c_1 \leqslant f''(u) \leqslant c_2 < \infty ,
\end{gather}
for some given positive constants $c_1$ and $c_2$. Additionally, it will become apparent later on that the initial datum will be required to be of a bounded variation and satisfies a one-sided Lipschitz condition, that is
\begin{gather}\label{assum:BV}
u_0' \in L^1(\R) \quad \text{and} \quad M \eqdef \sup_{x \in \R} u_0'(x) < \infty .
\end{gather}
Further precisions on the initial data will be discussed, later on.
Note that the equation \eqref{rB} is  equivalent to   \eqref{rSCL} for smooth solutions. Indeed, one easily sees  that \eqref{rSCL} is formally recovered from \eqref{rB}  by applying the elliptic operator $\text{Id}-\ell^2\partial^2_x$. Clearly the analysis of  \eqref{rB}, and thus \eqref{rSCL}, hinges upon a comprehensive study of the term $P^\ell$ and its behavior in suitable functional spaces. A primary important observation here is the validity of the identity  
\begin{equation*} 
P^\ell = \left(\text{Id}-\ell^2\partial^2_x \right)^{-1} \left( \half  f''(u^\ell) \left(\partial_x u^\ell\right)^2 \right) =  \half  f''(u^\ell) \left(\partial_x u^\ell\right)^2 \ast \mathfrak{G}_\ell ,
\end{equation*}
where
\begin{equation*}
\mathfrak{G}_\ell(x) \bydef  {\textstyle \frac{1}{2\ell} }   \exp\left( -{\textstyle \frac{|x|}\ell}\right) .
\end{equation*}  
Notice that  \eqref{assum:flux} entails the lower bound  
\begin{equation}\label{P:positive}
	P^\ell\geqslant 0,
\end{equation}
which will come in handy later on.

Let us now introduce the notion of solutions that we are concerned with in this paper.
\begin{definition}\label{WSDef} We say that $u^\ell \in L^\infty(\mathbb{R}^+; H^1) \cap \mathrm{Lip}(\mathbb{R}^+; L^2)$ is a weak dissipative solution of \eqref{rB} if it satisfies the equation in the $L^2$ sense, dissipates the energy in a weak sense
\begin{gather*} 
\partial_t \left( \half \left(u^{\ell}\right)^2 + \half \ell^2 \left(\partial_x u^{\ell}\right)^{ 2} \right) 
+ \partial_x \left(  K \left(u^{\ell}\right) + \half  \ell^2  f'\! \left(u^{\ell}\right)  \left(\partial_x  u^{\ell}\right)^2 + \ell^2  u^{\ell}  P^\ell  \right)  \leqslant 0,
\end{gather*}
where $K'(u)=uf'(u)$, and is  right continuous in $H^1$, that is to say 
\begin{equation*} 
\lim_{\substack{t\to t_0\\  t> t_0}} \left\| u^\ell(t,\cdot)  -  u^\ell(t_0,\cdot)\right\|_{H^1} = 0,
\end{equation*}
for all $t_0 \geqslant 0$.
In particular, the solution is required to satisfy the initial    condition $u^\ell(0,\cdot)=u_0$ in the sense of the $H^1$ norm. 
\end{definition}
 
One way to establish the existence of global weak dissipative solutions to \eqref{rB} can be performed via  introducing a viscosity term, leading to  the equation
\begin{equation}\label{rBep}
\partial_t u^{\ell,\varepsilon} + \partial_x f(u^{\ell,\varepsilon}) + \ell^2  \partial_x P^{\ell,\varepsilon} = \varepsilon  \partial_x^2 u^{\ell,\varepsilon}, \qquad \quad  P^{\ell,\varepsilon} - \ell^2  \partial _x^2 P^{\ell,\varepsilon} = \half  f''(u^{\ell, \varepsilon}) \left(\partial_x u^{\ell,\varepsilon}\right)^2,
\end{equation}
supplemented with the regularized initial datum $u^{\ell,\varepsilon}(t,\cdot) = u^{\varepsilon}_0 \bydef u_0 \ast \varphi_\varepsilon$, where $\varphi _\varepsilon$ stands for the standard one-dimensional Friedrich's mollifier.
Additionally,
following \cite{CHK05,G23,XZ00}, one can show that, up to an extraction of a subsequence, solutions of \eqref{rBep} converge to dissipative solutions of \eqref{rB}  as $\varepsilon \to 0$ in the sense that
\begin{equation}\label{vanishingviscosity}
	u^{\ell,\varepsilon} \to u^\ell \quad  \text{in }  C_{\text{loc}}([0,\infty) \times \R)  .
\end{equation} 
 In another word,  solutions of \eqref{rB} can be constructed as accumulation points of the family of regularized solutions $u^{\ell,\varepsilon}$ as $\varepsilon \to 0$.  As a result, the following theorem holds.

\begin{thm} \label{thm:existence}
Consider an initial datum $u_0\in H^1(\mathbb{R})$ and assume that the flux is uniformly convex in the sense of \eqref{assum:flux}. 
Then, for any $\ell>0$, there exists a global weak dissipative solution $u^\ell \in L^\infty ([0,\infty ), H^1(\mathbb{R})) \cap C([0,\infty ) \times \mathbb{R})$ of \eqref{rB} in the sense of Definition \ref{WSDef} satisfying the following:
\begin{itemize}
\item For any $T>0$, any bounded set $ [a,b] \subset \mathbb{R}$ and $\alpha \in [0,1)$ there exists $C=C(\alpha, T, a, b, \ell)>0$ such that 
\begin{equation*} 
\int_0^T \int_a^b \left(\left|\partial_t u^\ell\right|^{2+\alpha} + \left| \partial_x u^\ell\right|^{2+\alpha} \right) \mathrm{d}x\,  \mathrm{d}t \leqslant C . 
\end{equation*}
\item The   one-sided Ole\u{\i}nik inequality
\begin{equation*} 
\partial_x u^{\ell}(t,x) \leqslant \frac{1}{  \frac{ c_1 t}{2}  +  \frac 1M} \quad \text{a.e. } (t,x) \in (0,\infty) \times \R,
\end{equation*}
where $M=\sup_{x\in \mathbb{R}} u_0'(x) \in (0,\infty].$ 
\end{itemize}
Moreover, if the initial datum satisfies \eqref{assum:BV}, then it holds that 
\begin{equation*} 
\left\| u^{\ell}(t) \right\|_{L^\infty}  \leqslant  \left\| \partial_x u^{\ell} (t) \right\|_{L^1}  \leqslant \left\| u_0' \right\|_{L^1} \left( \frac{c_1  M  t}{2}  +  1 \right)^{\frac{2  c_2}{c_1} },
\end{equation*}
for all $t\geqslant 0$.
\end{thm}

The proof of Theorem \ref{thm:existence} is presented in \cite{G23} using another approximated equation involving a cut-off function, rather than the viscous approximation \eqref{rBep}. However, the same elements of proof therein   remain valid for the viscous approximation, too. See also \cite{CHK05,XZ00} for the vanishing viscosity limit for the Camassa--Holm equation. 

Henceforth, we agree that $u^\ell$ is a dissipative solution obtained from a vanishing viscosity process, though,   we believe that it  possible to show that this is actually the unique dissipative solution to \eqref{rB}, as discussed in our next remark.

\begin{rem}

Dafermos \cite{D11} proved, following the characteristics, that dissipative solutions of the Hunter--Saxton equation are unique. In the same spirit, the uniqueness of dissipative solutions to the Camassa--Holm equation \eqref{CH} has been proved in \cite{J16}. Additionally, a different proof has been established recently in \cite{G23a}. Although  we do not address   such an issue in this paper,  we believe that, following the same arguments from \cite{D11,G23a,J16}, one could prove the uniqueness of dissipative solutions to the regularized equation \eqref{rB}, too. 
\end{rem}

As previously discussed in the introduction, given any     solution of \eqref{rB} by Theorem \ref{thm:existence}, the next natural question to be asked is about its behavior as $\ell$ tends to zero. In \cite{G23}, the first author constructed global dissipative solutions to \eqref{rB} converging to the unique entropy solution of \eqref{SCL} in $ L^\infty_{\text{loc}}(\mathbb{R}^+;L^p_{\text{loc}}(\R)),$ for any   $p\in [1,\infty)$.

  The main result of this paper improves the preceding convergence by showing that it holds globally in space and by also obtaining a precise rate of convergence. This is the content of the next theorem.

\begin{thm}\label{main:thm}
Let $u^\ell$ be any solution, given by Theorem \ref{thm:existence}, of \eqref{rB}   with a uniformly convex flux \eqref{assum:flux} and an initial datum $u_0\in H^1(\mathbb{R})$ satisfying \eqref{assum:BV}. Consider, moreover,  $u$ to be the unique entropy solution of the scalar conservation law  \eqref{SCL} with the same initial datum $u_0$. Then, for any $T>0$, there exists a constant 
\begin{equation*}
C = C\left(T,\norm{u_0}_{H^1(\R)},  \norm{u_0'}_{L^1(\R)}, M, c_1,c_2 \right) > 0,
\end{equation*}
such that
\begin{equation}\label{uell-u}
\norm{u^\ell - u}_{L^\infty ([0,T];L^p(\R))} \leqslant C \ell^\frac{1}{2 p},
\end{equation}
for any $\ell \in (0,1]$ and $p \in [1, \infty)$.
\end{thm} 

\begin{rem}
It will be apparent in the proof below that the convergence \eqref{uell-u} can be improved to hold in Sobolev spaces which scale below the $BV$ regularity. However, we have chosen to only state the convergence in Lebesgue spaces for the mere sake of simplicity.
\end{rem}  
In fact, we will prove later on (see Theorem \ref{main:thm:general}) a slightly stronger version of Theorem  \ref{main:thm} where  solely the  initial datum  of the regularized equation \eqref{rB} is assumed to belong to $\dot H^1(\mathbb{R})$. In that case,  the initial datum of \eqref{SCL} is not required to belong to $\dot H^1(\mathbb{R})$ whereas the initial datum of \eqref{rB} may depend on $\ell$ and have a growth rate of its $\dot H^1$-norm at most of order $\ell^{-1}$ as $\ell\to 0$.  

Below, we briefly  discuss the main challenges and comment on the strategy of our proof of the main theorem.

\subsection*{Methodology and idea of the proof}  A naive way to study the convergence of $u^\ell-u$ to zero would be by   performing $L^p$ energy estimates directly on the equation 
\begin{equation}\label{fluctuation:equa}
	\partial_t (u^\ell-u) + \partial_x\left(  f(u^\ell)-f(u)\right) + \ell^2  \partial_x P^\ell = 0.
\end{equation}   
Such a direct attempt to analyze the fluctuations $u^\ell-u$ probably would  not be efficient and the most drawback here would be the potential  instability of the term $\ell^2 \partial_x P^\ell$, as $\ell \to 0$, in any Lebesgue space. Of course, neither the stability of the nonlinear term  $\partial_x\left(  f(u^\ell)-f(u)\right)$ in Lebesgue spaces is clear to be  under control in the case of weak solutions.

The proof that we are going to present in this paper consists in first studying \eqref{fluctuation:equa} in a low regularity space, namely in a $L^\infty_{\mathrm{loc}}(\mathbb{R}^+;\dot W^{-1,1}(\mathbb{R}))$-like space. Once this is done, the convergence in Lebesgue spaces of the fluctuations $u^\ell-u$ will be achieved by a direct interpolation argument, seen that both $u^\ell$ and $u$ enjoy some additional regularity --- the $BV$ bound, to be more precise. This strategy of proof draws insight from the method introduced in \cite{AH23} to study the stability of Yudovich solutions to the two-dimensional Euler equations. Technically    speaking, the  idea consists in taking care of the high and low frequencies of the $L^2_x$-norm separately: the low frequencies of the fluctuations will converge to zero (with a certain rate) due to the convergence in low regularity spaces, whereas the high frequencies are just uniformly bounded due to the additional $BV$ regularity. This paradigm of proof will be implemented here in $L^1_x$-based spaces instead of $L^2_x$  in order to obtain a better rate of convergence.

Thus, a milestone in our approach is based upon the convergence of an anti-derivative of the fluctuations in $L^\infty_{\mathrm{loc}}(\mathbb{R}^+;L^1(\mathbb{R}))$ which is the subject of Section \ref{section:convergence:low}.  A crucial gain in the analysis of the equation \eqref{fluctuation:equa} in $\dot W^{-1,1}_x$ is that we will be solely seeking the stability of $\ell^2    P^\ell $, rather than $\ell^2  \partial_x P^\ell $, in Lebesgue spaces. As we shall prove in Section \ref{section:decay:estimate} later on,  the term $\ell^2 P^\ell$, which is equivalent to $\ell^2|\partial_x u^{\ell}|^2$, enjoys a decay rate of order $\ell  $ in $L^1_{\mathrm{loc}}(\mathbb{R}^+; L^1(\mathbb{R}))$. This is a consequence of a careful analysis, improving on some results from \cite{G23}, and is based on a step-by-step argument (of a bootstrap-type) leading to the aforementioned rate of convergence. 

For clarity, we  point out   that   this roadmap of proof will be conducted on regularized equations; the solutions of which are sufficiently regular to fulfill  all the requirements in our computations and estimates   which are close to the solutions $u^\ell$ and $u$. This will be detailed in Section \ref{section:UB} along side with all the a priori bounds on the regularized solutions. In the end, the proof of Theorem \ref{main:thm} will be outlined in Section \ref{proof:main:thm}.

\subsubsection*{Notations}  Allow us now to introduce some notations that will be routinely used throughout the paper. Given two positive quantities $A$ and $B$, we will often write   $A \lesssim B $ instead of   $A \leqslant C B$ when the dependence on the generic constant $C>0$   is  not of a substantial impact. Moreover, we will sometime use the notation $A \lesssim_\delta B $ to emphasize that the generic constant in that estimate  depends on the some parameter $\delta$, which could blow up when $\delta$ approaches  some critical values.

\section{Uniform estimates}\label{section:UB}

This section is devoted to establishing all the primary lineup of  bounds on $u$ and $u^\ell$, uniformly with respect to the parameter $\ell \in (0,1]$. This is obtained as a consequence of a regularization procedure, made by adding the viscosity dissipation  $\varepsilon \partial_x^2$  to the equations \eqref{SCL} and \eqref{rB} and by smoothing out the initial datum. More precisely, we approximate these equations  by 
\begin{equation}\label{RSC1}
\partial_t u^{\varepsilon} + \partial_x  f(u^{\varepsilon})   = \varepsilon  \partial_x^2 u^{\varepsilon}
\end{equation}
and
\begin{equation}\label{RB1}
\partial_t u^{\ell,\varepsilon} + \partial_x  f(u^{\ell,\varepsilon})  + \ell^2  \partial_x P^{\ell,\varepsilon} = \varepsilon  \partial_x^2 u^{\ell,\varepsilon}, \qquad \quad  P^{\ell,\varepsilon} - \ell^2  \partial _x^2 P^{\ell,\varepsilon} = \half  f''(u^{\ell,\varepsilon}) \left(\partial_x u^{\ell,\varepsilon}\right)^2
\end{equation} 
respectively, where  $\varepsilon\in (0,1]$ and  both regularized equations are supplemented with the smooth initial data   
\begin{equation*}
u^{\varepsilon}|_{t=0}= u^{\varepsilon}_0 \eqdef \varphi_\varepsilon \ast u_0, \qquad	u^{\ell,\varepsilon}|_{t=0}= u^{\ell,\varepsilon}_0 \eqdef \varphi_\varepsilon \ast u_0^\ell, 
\end{equation*} 
where $(\varphi_\varepsilon )_{\varepsilon\in (0,1]}$ stands for the usual one-dimensional mollifier.  In particular, we emphasize that the regularized solutions $u^{\ell,\varepsilon}$ and $u^\varepsilon$ enjoy enough regularities that will allow us to perform all the computations in this section.

As previously emphasized, we shall prove a slightly stronger version of Theorem \ref{main:thm} where the initial data enjoy weaker assumptions. More precisely, henceforth, the original equations \eqref{SCL} and \eqref{rB} will be supplemented with the possibly different initial datum $u_0$ and $u_0^\ell$, respectively, with the emphasis that the  case of Theorem \ref{main:thm} is recovered by setting $u_0^\ell =u_0$ without any substantial change in our arguments below.
 
 In what follows,   we stick to the assumptions that   \begin{equation}\label{u0:COND}
 	u_0 \in L^2(\R) \cap BV(\R) \quad \text{and} \quad  M \eqdef \sup_{x,y \in \R,\ x \neq y} \frac{u_0(x)-u_0(y)}{x-y} < \infty.
 \end{equation}
Moreover, we consider a family $(u_0^\ell)_{\ell > 0}$ of smooth initial data satisfying
\begin{equation}\label{u0ell1}
\|u_0^\ell\|_{L^2(\R)} \leqslant \|u_0\|_{L^2(\R)}, \qquad \|\partial_x u_0^\ell\|_{L^1(\R)} \leqslant \|u_0\|_{BV(\R)}, \qquad  \sup_{x\in \R}\partial_x u_0^\ell(x) \leqslant M
\end{equation}
and 
\begin{equation}\label{u0ell2}
\ell \|\partial_x u_0^\ell\|_{L^2(\R)} \lesssim_{u_0} 1 ,
\end{equation} 
for any $\ell\in (0,1]$. Notice that, under the first two conditions in \eqref{u0ell1},  the  weak convergence of $u_0^\ell$   to $u_0$, as $\ell\to 0$, is equivalent to its strong convergence in $L^p(\mathbb{R})$, for any $p\in [2,\infty)$. This is a direct consequence of Fatou's lemma.

In Sections \ref{section:uniform:1} and \ref{section:uniform:2} below, we outline the elements of proof of the energy bounds as well as the $\dot W^{1,1}$ control on the regularized solutions uniformly with respect to the parameters $\ell, \varepsilon \in (0,1]$. 
 Note that similar findings have been established by the first author in \cite{G23} using a different approach, based on a cut-off argument. Here, we show that the same final bounds on $u^\ell$ can be achieved by a classical viscosity-regularization procedure, as well.

\subsection{Uniform bounds on $u^{\ell,\varepsilon}$}\label{section:uniform:1} We begin with establishing  the $H^1$-energy bound on $u^{\ell,\varepsilon}$.  To that end, introducing
 \begin{equation*} 
q^{\ell,\varepsilon} \eqdef \partial_x u^{\ell,\varepsilon}
\end{equation*}
and differentiating \eqref{RB1} with respect to the space variable we obtain that 
\begin{equation}\label{rBep_x}
\partial_t q^{\ell,\varepsilon}  +  f'(u^{\ell,\varepsilon})  \partial_x q^{\ell,\varepsilon} + \half  f''(u^{\ell,\varepsilon})  (q^{\ell,\varepsilon})^2 + P^{\ell,\varepsilon} = \varepsilon  \partial_x^2 q^{\ell,\varepsilon}.
\end{equation}
Thus, it follows, by summing the resulting equations of multiplying \eqref{RB1} by $u^{\ell,\varepsilon}$ and \eqref{rBep_x} by $\ell^2 q^{\ell,\varepsilon}$, that  
\begin{equation*}  
\begin{aligned}
	\half	\partial_t  &\left(\left(u^{\ell,\varepsilon}\right)^2 
	+  \ell^2 \left(q^{\ell,\varepsilon}\right)^{ 2}\right)
+  \partial_x \left( K\! \left(u^{\ell,\varepsilon}\right) +  \half  \ell^2  f'\! \left(u^{\ell,\varepsilon}\right) \left(q^{\ell,\varepsilon}\right)^2 +  \ell^2  u^{\ell,\varepsilon}  P^{\ell,\varepsilon} \right)\\ 
 & \qquad\qquad  -   \varepsilon  \ell^2  \partial_x \left( q^{\ell,\varepsilon}   \partial_x q^{\ell,\varepsilon} \right)  -    \varepsilon    \partial_x \left( u^{\ell,\varepsilon}   \partial_x u^{\ell,\varepsilon} \right)   = - \varepsilon  \ell^2  \left(\partial_x q^{\ell,\varepsilon}\right)^2  - \varepsilon    \left( q^{\ell,\varepsilon}\right)^2   ,
\end{aligned}
\end{equation*}  
where $K'(u)=uf'(u)$.
Hence, integrating in time and space   and using \eqref{u0ell2} we obtain the energy bound 
\begin{equation} \label{eneequation} 
	\begin{aligned}
		\int_\R \left( |u^{\ell,\varepsilon}|^2  +  \ell^2  |q^{\ell,\varepsilon}|^2 \right) \ud x  + 2  \varepsilon  & \ell^2 \int_0^t \int_\R |\partial_x q^{\ell,\varepsilon}|^2 \ud x\,  \ud t + 2  \varepsilon    \int_0^t \int_\R | q^{\ell,\varepsilon}|^2 \ud x\,  \ud t\\ 
&= \int_\R \left( |u_0^\varepsilon|^2  +  \ell^2  \left|\partial_x u_0^\varepsilon\right|^2 \right) \ud x\
\lesssim_{u_0} 1.
	\end{aligned}
\end{equation}
Subsequently, integrating the second equation on the right-hand side of \eqref{RB1} with respect to the space variable, we obtain, in view of \eqref{assum:flux}, that
\begin{align}\label{ell2P}
\half  c_1  \ell^2  \|q^{\ell,\varepsilon}\|_{L^2}^2 \leqslant\
\int_\mathbb{R} \ell^2  P^{\ell,\varepsilon} \, \mathrm{d}x \leqslant \half  c_2  \ell^2  \|q^{\ell,\varepsilon}\|_{L^2}^2   \lesssim_{u_0} 1. 
\end{align}
The preceding bounds will come in handy, later on. 

The next lemma produces the Ole\u{\i}nik inequality for $u^{\ell,\varepsilon}$. Again, we outline its proof thereafter for completeness.
\begin{lem}\label{Lem:Oleinik}
 Assume $u_0 \in L^2(\mathbb{R})$ and $u_0^\ell$ satisfying
\begin{equation*}
\sup_{x\in \mathbb{R}} \partial_x u_0^\ell(x) \leqslant M \bydef \sup_{x,y \in \R,\ x \neq y} \frac{u_0(x)-u_0(y)}{x-y} \in (0,\infty].
\end{equation*}
Assume that the flux $f$ fulfills the uniform convexity condition  \eqref{assum:flux}. Then, it holds that 
\begin{equation}\label{Oleinikellep}
\partial_x u^{\ell,\varepsilon}(t,x) \leqslant \frac{1}{\frac{c_1  }{2} t  +  \frac{1}{M}} \qquad \text{for a.e. } (t,x) \in (0,\infty) \times \R.
\end{equation}
\end{lem}

\begin{proof} 
We begin with noticing that Theorem 2.1  from \cite{CE98} ensures the existence of at least one point $\xi(t) \in \R$ such that 
\begin{equation*}
h(t) \eqdef \sup_{x \in \R} q^{\ell,\varepsilon}(t,x) = q^{\ell,\varepsilon}(t,\xi(t)),
\end{equation*} 
where the function $h$ is locally Lipschitz and is governed by the equation 
\begin{equation*}
\frac{\ud h}{\ud t} (t) = \partial_t q^{\ell,\varepsilon}(t,\xi(t)), \quad \text{for all } t >0.
\end{equation*}
Since $q^{\ell,\varepsilon}(t,\cdot)$ reaches its maximum at $\xi(t)$, it then follows that 
\begin{equation*}
\partial_x q^{\ell,\varepsilon}(t,\xi(t)) = 0, \qquad \partial_x^2 q^{\ell,\varepsilon}(t,\xi(t)) \leqslant 0.
\end{equation*}
Accordingly, we deduce from  \eqref{assum:flux}, \eqref{rBep_x} and the fact that $P^{\ell,\varepsilon} \geqslant 0$ (which can be established in a similar way   to \eqref{P:positive}) that
\begin{equation*} 
\frac{\ud h}{\ud t} (t)  \leqslant - \frac{c_1}{2} \left(h(t)\right)^2.
\end{equation*}
At last, solving the preceding inequality   with the initial condition   $h(0) \leqslant M$ completes the proof of the lemma.   
\end{proof}

The next item in our agenda is to establish the $\dot W^{1,1}$ bound on $u^{\ell,\varepsilon}$. This is the content of the following lemma.

\begin{lem}\label{lemTV}   Assume that  \eqref{u0:COND}, \eqref{u0ell1} and \eqref{u0ell2} hold for some  $u_0$ and $u_0^\ell$.
 Further  assume that the flux $f$ fulfills the uniform convexity condition  \eqref{assum:flux}. Then, it holds, for all $t\geqslant 0$, that     
\begin{equation}\label{TVep}
   \left\| \partial_x u^{\ell, \varepsilon} (t) \right\|_{L^1(\mathbb{R})}  \leqslant \left\| u_0 \right\|_{BV(\mathbb{R})} \left( \frac{c_1  M  t}{2}  +  1 \right)^{  \frac{2c_2}{c_1}} ,
\end{equation} 
for any $ (\ell,\varepsilon) \in (0,1] \times (0,1] $.
\end{lem}
\begin{rem}
	Note that, by one-dimensional Sobolev embeddings, the bound stated in the preceding lemma implies the control 
\begin{equation}\label{Linfep}
\left\| u^{\ell, \varepsilon}  (t)\right\|_{L^\infty(\mathbb{R})}    \leqslant \left\| u_0 \right\|_{BV(\mathbb{R})}\left( \frac{c_1  M  t}{2}  +  1 \right)^{ \frac{2c_2}{c_1}},
\end{equation} 
for all $t\geqslant 0$.
\end{rem}

\begin{proof} 
We begin with introducing the function
\begin{equation}\label{Sdef}
S_\delta(q) \eqdef 
\begin{cases}
-q - \half \delta, & q \in (-\infty,-\delta), \\
{\textstyle \frac{1}{2\delta}} q^2, & q \in [-\delta, \delta], \\
q - \half \delta, & q \in (\delta, \infty),
\end{cases}
\end{equation}
 for any $\delta \in (0,1]$. Accordingly, it is readily seen that 
\begin{equation}\label{Bound:a}
	\left| q S_\delta(q  )- q  ^2 S_\delta'(q  ) \right| \leqslant  \delta S_\delta(q)
\end{equation}
and
\begin{equation}\label{|q|S2}
|q| \mathds{1}_{\{ |q| \geqslant \delta \}} \leqslant  S_\delta(q) + \half \delta \mathds{1}_{\{ |q| \geqslant \delta \}} , \qquad 
|q| \mathds{1}_{\{ |q| \geqslant \delta \}} \leqslant 2 S_\delta(q),
\end{equation} 
for any $q\in \mathbb{R}$. Additionally, one can easily check that 
$$S_\delta(q) \lesssim_\delta q^2,$$
 which, in view  of the energy bound \eqref{eneequation}, yields that
  $S_\delta(q^{\ell,\varepsilon})(t,\cdot) \in L^1(\R)$ for all $\delta,\ell,\varepsilon \in (0,1]$ and any $t \geqslant 0$. 
  
  Next, we want to establish a uniform bound,  with respect to the parameters $\ell$ and $\delta$, on the preceding  $L^1$ control of $S_\delta(q^{\ell,\varepsilon})(t,\cdot)$. To that end,  multiplying \eqref{rBep_x} by $S'_\delta(q^{\ell,\varepsilon})$ yields, in view of the second equation in \eqref{rBep}, that 
\begin{align}\nonumber
\partial_t S_\delta (q^{\ell,\varepsilon}) + \partial_x \left( f'(u^{\ell,\varepsilon}) S_\delta(q^{\ell,\varepsilon}) \right) 
&= f''(u^{\ell,\varepsilon}) \left( q^{\ell,\varepsilon} S_\delta(q^{\ell,\varepsilon})-  (q^{\ell,\varepsilon})^2 S_\delta'(q^{\ell,\varepsilon}) \right) - \ell^2  S_\delta'(q^{\ell,\varepsilon}) \partial_x^2 P^{\ell,\varepsilon}\\  \label{S(q)}
& \quad +\varepsilon \partial_x \left( S'_\delta(q^{\ell,\varepsilon}) \partial_x q^{\ell,\varepsilon} \right) - {\textstyle \frac{1}{\delta}} \mathds{1}_{|q^{\ell,\varepsilon}| \leqslant \delta} (\partial_x q^{\ell,\varepsilon})^2.
\end{align}
Next, writing 
\begin{align*}
- \ell^2  S_\delta'(q^{\ell,\varepsilon}) \partial_x^2 P^{\ell,\varepsilon} &=  \ell^2\left(    \mathds{1}_{\{q^{\ell,\varepsilon} < - \delta\}}   
 - { \frac{ q^{\ell,\varepsilon}}{\delta}} \mathds{1}_{\{|q^{\ell,\varepsilon}| \leqslant  \delta\}} 
-  \mathds{1}_{\{q^{\ell,\varepsilon} > \delta\}} \right) \partial_x^2 P^{\ell,\varepsilon} \\
& = \ell^2 \left(1 
 - \left( 1 + { \frac{ q^{\ell,\varepsilon}}{\delta}}\right)   \mathds{1}_{\{|q^{\ell,\varepsilon}| \leqslant  \delta\}}
- 2   \mathds{1}_{\{q^{\ell,\varepsilon} > \delta\}}  \right)  \partial_x^2 P^{\ell,\varepsilon} 
\end{align*}
and making use of the fact that 
\begin{equation*}
- \ell^2 \partial_x^2 P^{\ell,\varepsilon} = \half f''(q^{\ell,\varepsilon}) (q^{\ell,\varepsilon})^2 - P^{\ell,\varepsilon} \leqslant \half c_2 (q^{\ell,\varepsilon})^2 ,
\end{equation*}
which is a direct consequence of \eqref{assum:flux}, \eqref{RB1} and that $P^{\ell,\varepsilon} \geqslant 0$, we find that 
\begin{equation*}
	-\ell^2 \int_\R S_\delta'(q^{\ell,\varepsilon}) \partial_x^2 P^{\ell,\varepsilon}\, \ud x
	\leqslant {\textstyle \frac{c_2}{2}} \int_\R \left(  \left( 1 + { \frac{ q^{\ell,\varepsilon}}{\delta}} \right) \mathds{1}_{\{|q^{\ell,\varepsilon}| \leqslant  \delta\}} + 2 \mathds{1}_{\{q^{\ell,\varepsilon} > \delta\}} \right) (q^{\ell,\varepsilon})^2\, \ud x.
\end{equation*}
Therefore, simplifying the right-hand side by noticing, by definition of $S_\delta(q)$, that 
\begin{equation*}
	  \left( 1 + { \frac{ q^{\ell,\varepsilon}}{\delta}} \right)  |q^{\ell,\varepsilon}|^2\mathds{1}_{\{|q^{\ell,\varepsilon}| \leqslant  \delta\}} \leqslant 4\delta S_\delta(q^{\ell,\varepsilon})   
\end{equation*}
and, by virtue of \eqref{|q|S2} and Lemma \ref{Lem:Oleinik}, that
\begin{equation*}
	  \begin{aligned}
	  2	|q^{\ell,\varepsilon}|^2 \mathds{1}_{\{q^{\ell,\varepsilon} > \delta\}} 
	  	& \leqslant 2 \left(S_\delta(q^{\ell,\varepsilon}) + \half\delta \mathds{1}_{\{ |q^{\ell,\varepsilon}| \geqslant \delta \}} \right) q^{\ell,\varepsilon} \mathds{1}_{\{q^{\ell,\varepsilon} > \delta\}} 
	  	\\
	  	& \leqslant  \frac{2}{\frac{c_1  }{2} t  +  \frac{1}{M}}S_\delta(q^{\ell,\varepsilon}) +   2 \delta S_\delta(q^{\ell,\varepsilon}) ,
	  \end{aligned}
\end{equation*}
yields in the end that

 \begin{align*}
-\ell^2 \int_\R S_\delta'(q^{\ell,\varepsilon}) \partial_x^2 P^{\ell,\varepsilon}\, \ud x
&\leqslant   c_2 \left( 
3\delta  +    \frac{1}{ \frac{c_1 t}{2}   + \frac{1}{M}}  \right)  \int_\R  S_\delta(q^{\ell,\varepsilon})  \ud x.
\end{align*}
In account of that, integrating \eqref{S(q)} and utilizing \eqref{assum:flux} together with \eqref{Bound:a} to estimate the first term in its right-hand side leads to the bound
\begin{equation*}
\frac{\ud\ }{\ud t} \int_\mathbb{R} S_\delta(q^{\ell,\varepsilon})\, \mathrm{d}x \leqslant  \left( 
4 c_2 \delta   +  { \frac{ c_2}{ \frac{c_1 t}{2}   + \frac{1}{M}}} \right)    \int_\mathbb{R} S_\delta(q^{\ell,\varepsilon})\, \mathrm{d}x,
\end{equation*}
which implies, by Gr\"onwall's inequality, that   
\begin{equation*} 
\int_\mathbb{R} S_\delta(q^{\ell,\varepsilon})\, \mathrm{d}x \leqslant \mathrm{e}^{4 c_2 \delta t} \left( \frac{c_1  M  t}{2}  +  1 \right)^{\frac{2c_2}{c_1}  } \int_\mathbb{R}  |\partial_x u_0^\ell|\, \mathrm{d}x.
\end{equation*}
In the end, taking the limit $\delta \to 0$ and using the monotone convergence theorem with \eqref{u0ell1} completes the proof of the lemma. 
\end{proof}

\subsection{Uniform bounds on $u^{\varepsilon}$}\label{section:uniform:2}

For  $\varepsilon>0$, we consider here $u^\varepsilon$ the solution of the viscous scalar conservation law \eqref{RSC1}. For clarity, we are going to  recast the estimates on $u^\varepsilon$, without detailed justification.
The estimates presented below are well-known, we refer the reader to \cite{B00,D05,HR15,K70}  for additional details on the viscous approximation of hyperbolic equations.
Moreover, we emphasize that the same arguments presented in the preceding section can be employed here as well.

 The equation \eqref{RSC1} is globally well-posed, and $L^p$ norms satisfy the maximum principle 
\begin{equation}\label{Lp}
\|u^\varepsilon\|_{L^\infty([0,\infty), L^p(\R))} \leqslant \|u_0\|_{L^p (\R)},
\end{equation}
for any $p \in [1, \infty]$, as soon as the initial datum belongs to $L^p(\mathbb{R})$. In our case, due to the assumption that $u_0\in H^1(\mathbb{R})$ and the embedding $H^1(\mathbb{R}) \hookrightarrow L^2(\mathbb{R})\cap L^\infty(\mathbb{R})$, the bound \eqref{Lp} holds for any $p \in [2,\infty]$.
Moreover, we can show that the total variation of $u^\varepsilon$ is decreasing in time, i.e.,   
\begin{equation}\label{TVep0} 
\int_\R |\partial_x u^\varepsilon(t,x)|\, \ud x \leqslant \int_\R |\partial_x u_0^\varepsilon(x)|\, \ud x \leqslant \|u_0\|_{BV(\R)},
\end{equation}
for all $t\geqslant 0$.
Additionally, the solution of \eqref{RSC1} satisfies the one-sided Ole\u{\i}nik inequality 
\begin{equation}\label{Ol:vscl}
\partial_x u^\varepsilon(t,x) \leqslant \frac{1}{c_1  t +  \frac{1}{M}},
\end{equation}
for all $(t,x) \in (0,\infty) \times \R$, where  $M$ is defined in \eqref{u0:COND}.
Finally, one can show that, as $\varepsilon \to 0$ and up to a subsequence, we have the convergence 
\begin{equation}\label{vanishingviscosity0}
u^\varepsilon \to u \quad \mathrm{in}\ C([0,T];L^p_{\text{loc}}(\R)),
\end{equation}
for any $p\in [1,\infty)$ and any $T>0$, where $u$ is the unique entropy solution of the scalar conservation law \eqref{SCL}.

\section{Decay estimates   and Convergence}\label{section:decay}

\subsection{A crucial decay estimate} \label{section:decay:estimate} An important milestone in our proof of the strong compactness in Lebesgue spaces consists of the analysis of the same problem in a low regularity space. Interestingly, we are going to show that the convergence in a $\dot {W}^{-1,1}$-like space comes with a rate. This turns out to be a consequence of Proposition \ref{mu=0} below, which     can be considered as the main new contribution in this section, improving on previous results by the first author.
More precisely, in \cite{G23}, the first author proved that, for any $T>0$, any compact set $K \subset    \R$ and any $\alpha \in (0,\frac{2}{3})$, there exists a constant $C_{T,K, \alpha}>0$ such that 
\begin{align*}
\int_0^T \int_K \ell^2  P^{\ell,\varepsilon}\,  \mathrm{d}x\,  \mathrm{d}t &\leqslant \ell^\alpha  C_{T,K,\alpha},
\end{align*}
 for all $\varepsilon>0$, any $\ell \in (0,1]$. Here, we provide a twofold improvement on   the latter bound by showing its validity  for   $\alpha=1$ and   by also allowing the integral on $x$ to be effective  over the whole real line $  \R$ instead of a compact set $K$. This is the content of the next proposition which is in the crux of  the   key findings in this paper.  
\begin{pro}\label{mu=0}  
    Assume that  \eqref{u0:COND}, \eqref{u0ell1} and \eqref{u0ell2} hold for some  $u_0$ and $u_0^\ell$.
 Further  assume that the flux $f$ fulfills the uniform convexity condition  \eqref{assum:flux}. Then,   for all $T>0$, there exists $C_{T,u_0,c_1,c_2}>0$ such that  
\begin{align*} 
\int_0^T \int_\R \ell^2 P^{\ell,\varepsilon} \, \mathrm{d}x\,  \mathrm{d}t &\leqslant \ell C_{T,u_0,c_1,c_2},
\end{align*}
 for all $\ell \in (0, 1]$.
\end{pro} 
\begin{proof}
We proceed in four steps by establishing:
\begin{enumerate}
	\item A uniform bound on $ \ell^2  |q^{\ell,\varepsilon}|  P^{\ell,\varepsilon}$ in $L^1\left([0,\infty) \times \mathbb{R}\right)$, by a suitable energy estimate. 
	\item A uniform bound on $ \ell^2  |q^{\ell,\varepsilon}|^\beta  P^{\ell,\varepsilon}$ in $L^1\left([0,T) \times \mathbb{R}\right)$, for any $T\in (0,\infty]$ and any $\beta \in (\frac 23,1)$, by an interpolation argument.
	\item A uniform bound on $\ell ^2|q^{\ell,\varepsilon}|^{2+\beta}$ in   $L^1\left([0,T) \times \mathbb{R}\right)$, for any finite $T>0$, by estimating differently the case of the barely positive    values of $q^{\varepsilon,\ell}$ --- by the aid of the one-sided Ole\u{\i}nik inequality \eqref{Oleinikellep} ---, and the remaining range of its values --- by a constructive energy method.
	\item A decay rate for order $\ell $ for $ \ell^2    P^{\ell,\varepsilon}$ in $L^1\left([0,T] \times \mathbb{R}\right)$, for any finite time $T>0$,  concluding the proof by ``bootstrapping'' all the bounds shown to hold so far. 
\end{enumerate}

\subsection*{Step 1}
 Multiplying \eqref{rBep_x} by $\ell^2 |q^{\ell,\varepsilon}|$ we obtain that
\begin{equation*}
  \begin{aligned}
  	\partial_t \left(\half  \ell^2  q^{\ell,\varepsilon}  |q^{\ell,\varepsilon}| \right)
  	 &+ \partial_x \left(\half  \ell^2  f'(u^{\ell,\varepsilon})  q^{\ell,\varepsilon}  |q^{\ell,\varepsilon}| \right) + \ell^2  |q^{\ell,\varepsilon}|  P^{\ell,\varepsilon} 
  	\\
  	& \qquad \qquad = \varepsilon \ell^2 \partial_x \left( |q^{\ell,\varepsilon}| \partial_x q^{\ell,\varepsilon}\right) - \varepsilon \ell^2 \mathrm{sign} (q^{\ell,\varepsilon}) (\partial_x q^{\ell,\varepsilon})^2.
  \end{aligned}
\end{equation*}
Therefore, integrating in time and space, we find that 
\begin{align*}
\int_{(0,\infty) \times \mathbb{R}}\ell^2  |q^{\ell,\varepsilon}|  P^{\ell,\varepsilon} \, \mathrm{d}x\,  \mathrm{d}t 
\leqslant &  \half  \ell^2\left( \int_{\R}    q_0^{\ell,\varepsilon}  |q_0^{\ell,\varepsilon}| \, \mathrm{d}x \, \mathrm{d}t  - \lim_{t \to \infty} \int_{\R}  q^{\ell,\varepsilon}  |q^{\ell,\varepsilon}| \, \mathrm{d}x \, \mathrm{d}t \right)\\
&
+ \varepsilon  \ell^2  \int_{(0,\infty) \times \mathbb{R}}  (\partial_x q^{\ell,\varepsilon})^2\, \mathrm{d}x\,  \mathrm{d}t.
\end{align*}
Thus, in view of \eqref{eneequation}, we arrive at the bound 
\begin{equation}\label{a}
\int_{(0,\infty) \times \mathbb{R}} \ell^2  |q^{\ell,\varepsilon}|  P^{\ell,\varepsilon} \, \mathrm{d}x\,  \mathrm{d}t \lesssim_{u_0} 1 ,
\end{equation}
for all $\varepsilon >0$ and $\ell \in (0,1]$.

\subsection*{Step 2} Let  $k \in \mathbb{N}^\ast$ be fixed and we introduce 
$$\beta \bydef  \frac{2k}{2k+1} \in \left [\tfrac 23,1\right ).$$ 
For any $q \in \R$, we define $q^\beta = (q^{2 k})^{\frac{1}{2k+1}} \geqslant 0$.
Accordingly, we write by H\"older's inequality that  
\begin{equation*}
\int_0^T \int_\R \ell^2  (q^{\ell,\varepsilon})^\beta  P^{\ell,\varepsilon} \, \mathrm{d}x \, \mathrm{d}t \leqslant  \left( \int_0^T \int_\R \ell^2  |q^{\ell,\varepsilon}|  P^{\ell,\varepsilon} \,  \mathrm{d}x \, \mathrm{d}t \right)^\beta \left(\int_0^T \int_\R \ell^2  P^{\ell,\varepsilon} \, \mathrm{d}x \, \mathrm{d}t \right)^{1-\beta}.
\end{equation*}
Therefore, it follows from \eqref{a} that  
\begin{equation}\label{b}
\int_0^T \int_\R \ell^2  (q^{\ell,\varepsilon})^\beta  P^{\ell,\varepsilon} \,  \mathrm{d}x\,  \mathrm{d}t \lesssim_{u_0} \left(\int_0^T \int_\R \ell^2  P^{\ell,\varepsilon} \, \mathrm{d}x\,  \mathrm{d}t \right)^{1-\beta}.
\end{equation}
In view of \eqref{ell2P}, the preceding control provides us with a uniform bound on $\ell^2  (q^{\ell,\varepsilon})^\beta  P^{\ell,\varepsilon}$. However, in the next step of the proof, we shall make use of the more precise estimate \eqref{b}  in order to obtain the desired decay, as $\ell \to 0$, of the term in its right-hand side.

\subsection*{Step 3} Our aim in this step is to obtain a bound on $\ell ^2|q^{\ell,\varepsilon}|^{2+\beta}$ in $ L^1_{t,x}$.
We   begin with   noting that 
\begin{equation*}
   \begin{aligned}
   	\int_0^T \int_{\mathbb{R}}   |q^{\ell,\varepsilon}|^{2+\beta}\mathds{1}_{\{q^{\ell,\varepsilon} \geqslant -1\}}\, \ud x\, \ud t 
   	&\leqslant    \int_0^T \int_{\mathbb{R}}  |q^{\ell,\varepsilon}|^{\beta+2}\mathds{1}_{\{q^{\ell,\varepsilon} \geqslant 0\}} \, \ud x\, \ud t   
   	\\
   	&\quad + \int_0^T \int_{\mathbb{R}}   |q^{\ell,\varepsilon}|^{\beta+2}\mathds{1}_{\{ |q^{\ell,\varepsilon} |\leqslant 1\}} \, \ud x\, \ud t 
   	\\
   		&\leqslant   \left( M^{1+\beta } + 1 \right)  \int_0^T \int_{\mathbb{R}}  |q^{\ell,\varepsilon}| \, \ud x\, \ud t  ,
   \end{aligned}	
\end{equation*}
where we employed the one-sided Ole\u{\i}nik inequality \eqref{Oleinikellep} in the second estimate. Thus, we deduce, in view of \eqref{TVep}, that 
\begin{equation}\label{bound:step3:A}
		\ell^2\int_0^T \int_{\mathbb{R}}   |q^{\ell,\varepsilon}|^{2+\beta}\mathds{1}_{\{q^{\ell,\varepsilon} \geqslant -1\}} \, \ud x\, \ud t  \lesssim _{T,u_0,c_1,c_2} 	\ell^2,
\end{equation}
for any finite time $T>0$. Now we take care of the the case where $ q^{\ell,\varepsilon}<-1$. To that end, we introduce the function 
\begin{equation*}
	G(q) \bydef \frac{1}{1+\beta }
	\begin{cases}
|q|^{1+\beta }, & q \in (-\infty,-1], \\
(1-\beta) q^3 + (2-\beta) q^2 , & q \in [-1, 0], \\
0, & q \in [0, \infty),
\end{cases}
\end{equation*}
and we emphasize that, by a straightforward computation, one can show that 
\begin{equation}\label{G:positive}
	0\leqslant G(q) \leqslant 2|q|^{1+\beta}  , \quad \text{for all  }\  q \in \mathbb{R},
\end{equation}
and that $G$ is twice differentiable almost every where with 
\begin{equation}\label{G':bound}
	0 \leqslant  -G'(q) \leqslant 5|q|^{\beta} , \quad \text{for all } \ q\in \mathbb{R}
\end{equation}
and
\begin{equation}\label{G:sec:positive}
	G''(q) \geqslant 0, \quad \text{for almost all } \ q\in \mathbb{R}.
\end{equation}
Now, multiplying \eqref{rBep_x} by $G'(q^{\ell,\varepsilon})$ and rearranging the resulting terms yields that 
\begin{equation}\label{A:equation}
	\begin{aligned}
	f''(u^{\ell,\varepsilon})	\mathcal{A} (q^{\ell,\varepsilon})
		& = - \partial_t \left(G(q^{\ell,\varepsilon})\right) -  \partial_x\left( f'(u^{\ell,\varepsilon})G(q^{\ell,\varepsilon}) \right) - G'(q^{\ell,\varepsilon}) P^{\ell,\varepsilon}
		\\
		&\quad  + \varepsilon  \partial_x\left(   G'(q^{\ell,\varepsilon})  \partial_x q^{\ell,\varepsilon}\right) - \varepsilon G''(q^{\ell,\varepsilon}) (\partial_x q^{\ell,\varepsilon})^2,
	\end{aligned}
\end{equation}
where we denoted 
\begin{equation*}
	\begin{aligned}
		\mathcal{A} (q^{\ell,\varepsilon})
		&\bydef  \half (q^{\ell,\varepsilon})^2 G'(q^{\ell,\varepsilon}) - q^{\ell,\varepsilon} G(q^{\ell,\varepsilon})
	= \frac{(1-\beta)}{2(1+\beta)}
	\begin{cases}
|q^{\ell,\varepsilon}|^{2+\beta }, & q \in (-\infty,-1], \\
|q^{\ell,\varepsilon}|^4   , & q \in [-1, 0], \\
0, & q \in [0, \infty),
\end{cases}
	\end{aligned}
\end{equation*}
whereby, due to the convexity condition \eqref{assum:flux}, it is readily seen that 
\begin{equation*}
	f''(u^{\ell,\varepsilon})	\mathcal{A} (q^{\ell,\varepsilon})
	\geqslant  \frac{c_1(1-\beta)}{2(1+\beta)} |q^{\ell,\varepsilon}|^{2+\beta } \mathds{1}_{\{q\leqslant-1 \}} .
	\end{equation*}
Therefore, integrating \eqref{A:equation} in space and time and making use of \eqref{G:positive}, \eqref{G':bound} and \eqref{G:sec:positive}, dropping the terms having a good sign, infers that 
\begin{equation*}
	\begin{aligned}
		\int_0^T \int_{\mathbb{R}} |q^{\ell,\varepsilon}|^{2+\beta } \mathds{1}_{\{q\leqslant-1 \}}  \, \ud x\, \ud t
		& \lesssim_{\beta,c_1}  \int_{\mathbb{R}} G(q^{\ell,\varepsilon}_0)\, \ud x - \int_0^T \int_{\mathbb{R}} G'(q^{\ell,\varepsilon})P^{\ell,\varepsilon} \, \ud x\, \ud t
		\\
		& 
		\lesssim_{\beta,c_1}  \int_{\mathbb{R}} |q^{\ell,\varepsilon}_0|^{1+\beta} \, \ud x + \int_0^T \int_{\mathbb{R}} |q^{\ell,\varepsilon}|^{\beta}P^{\ell,\varepsilon} \, \ud x\, \ud t.
	\end{aligned}
\end{equation*}
Since $\beta\in [2/3,1)$, employing H\"older's inequality, we end up with 
\begin{align*}
	\ell^2\int_0^T \int_{\mathbb{R}} |q^{\ell,\varepsilon}|^{2+\beta } \mathds{1}_{\{q\leqslant-1 \}}\, \ud x\, \ud t  
&\lesssim_{\beta,c_1}
  \ell^{2-2\beta} \norm{\partial_x u_0^\ell}_{L^1}^{1-\beta} \left( \ell \norm{\partial_x u_0^\ell}_{L^2}\right)^{2\beta}\\ 
&\qquad + \ell^2\int_0^T \int_{\mathbb{R}} |q^{\ell,\varepsilon}|^{\beta}P^{\ell,\varepsilon} \, \ud x\, \ud t.
\end{align*}  
All in all, in combination with \eqref{bound:step3:A} and by employing \eqref{b} and \eqref{u0ell2}, we deduce that 
\begin{equation} \label{q:beta+2:bound}
	\begin{aligned}
	\ell^2	\int_0^T \int_{\mathbb{R}} |q^{\ell,\varepsilon}|^{2+\beta }  \, \mathrm{d}x\,  \mathrm{d}t
		& \lesssim _{T,u_0,c_1,c_2, \beta}   \ell^{2-2\beta} +\left(\int_0^T \int_\R \ell^2  P^{\ell,\varepsilon} \, \mathrm{d}x\,  \mathrm{d}t \right)^{1-\beta}.
	\end{aligned} 
\end{equation}
	
	\subsection*{Step 4} We are now in position to conclude the proof of the proposition. To that end, we first write by H\"older inequality that
	\begin{equation*} 
	\begin{aligned}
		\int_0^T \int_\R \ell^2 |q^{\ell,\varepsilon}|^{2}  \, \ud x  \, \ud t
&\leqslant  \ell^\frac{2 \beta}{1+\beta} \left(\int_0^T \int_\R |q^{\ell,\varepsilon}|  \, \ud x  \, \ud t\right)^\frac{\beta}{1+\beta} \left(\int_0^T \int_\R \ell^2  (q^{\ell,\varepsilon})^{2+\beta}  \, \ud x  \, \ud t\right)^\frac{1}{1+\beta}\\
&\lesssim_{T,u_0,c_1,c_2,\beta} \ell^\frac{2 \beta}{1+\beta}  \left(\int_0^T \int_\R \ell^2  (q^{\ell,\varepsilon})^{2+\beta}  \, \ud x  \, \ud t\right)^\frac{1}{1+\beta}.
	\end{aligned}
\end{equation*}
Therefore, it follows by virtue of \eqref{ell2P} that
\begin{equation*}
   \begin{aligned}
   	  \left(
\int_0^T \int_\R \ell^2  P^{\ell,\varepsilon}  \, \ud x  \, \ud t\right)^{1+\beta} 
& \lesssim_{c_2} \left( \int_0^T \int_\R \ell^2 |q^{\ell,\varepsilon}|^{2}  \, \ud x  \, \ud t\right)^{1+\beta}
\\
&\lesssim_{T,u_0,c_1,c_2,\beta} \ell^{2 \beta}    \int_0^T \int_\R \ell^2  (q^{\ell,\varepsilon})^{2+\beta}  \, \ud x  \, \ud t,
   \end{aligned}	
\end{equation*}
whereby we deduce, by substituting    \eqref{q:beta+2:bound}, that 
\begin{equation*}
	\left(
\int_0^T \int_\R \ell^2  P^{\ell,\varepsilon}  \, \ud x  \, \ud t\right)^{1+\beta} \lesssim_{T,u_0,c_1,c_2,\beta}      \ell^{2} + \ell^{2\beta }  \left(\int_0^T \int_\R \ell^2  P^{\ell,\varepsilon}  \, \ud x  \, \ud t \right)^{1-\beta}.
\end{equation*}
Hence, writing, by Young inequality, for any $\lambda>0$, that
\begin{equation*}
	\left(
\int_0^T \int_\R \ell^2  P^{\ell,\varepsilon}  \, \ud x  \, \ud t\right)^{1+\beta} \lesssim_{T,u_0,c_1,c_2,\beta}     \ell^{2} + C_\lambda\ell ^{1+\beta }  + \lambda \left(\int_0^T \int_\R \ell^2  P^{\ell,\varepsilon}  \, \ud x  \, \ud t \right)^{1+\beta}
\end{equation*}
and choosing $\lambda$ as small as it is needed to absorb the last term in the right-hand side by the right hand side concludes the proof of the proposition.
\end{proof}

\subsection{Convergence in a low regularity space}\label{section:convergence:low}

In this section, we establish a stability estimate for the difference 
\begin{equation*}
	w^{\ell,\varepsilon} \eqdef u^{\ell,\varepsilon} - u^{\varepsilon}
\end{equation*}
in a low regularity space. This will be done by particularly studying the evanescence of the   fluctuation 
\begin{equation*}
	 \zeta  ^{\ell,\varepsilon}(t,x) \eqdef \int_{-\infty }^x w^{\ell,\varepsilon}(t,y) \, \ud y, \qquad (t,x) \in (0,\infty)\times \mathbb{R},
\end{equation*}
in $L^\infty_tL^1_x$, as $\ell $ tends to zero, which   crucially builds upon the decay estimate from the preceding proposition. This is the content of the next proposition.
\begin{pro}
 Assume that  \eqref{u0:COND}, \eqref{u0ell1} and \eqref{u0ell2} hold for some  $u_0$ and $u_0^\ell$ satisfying
\begin{equation*}
		x\mapsto \zeta_0^\ell(x) \eqdef \int_{-\infty}^{x } \big (u_0^\ell(y) - u_0(y)\big )\,\ud y  \in L^1(\mathbb{R}),
	\end{equation*}
	for any fixed $\ell \in (0,1]$. Further  assume that the flux $f$ fulfills the uniform convexity condition  \eqref{assum:flux}. Then, it holds that 
 \begin{equation*}
 	 \zeta   ^{\ell, \varepsilon} \in  L^\infty ([0,T];L^1(\mathbb{R}))
 \end{equation*}
 for all $T>0$,  with    
\begin{equation}\label{wep}
\norm { \zeta  ^{\ell,\varepsilon}}_{L^\infty([0,T]; L^1(\R))} \lesssim_{T,u_0,c_1,c_2} \int_\R |\zeta  ^{\ell}_0 (x )|\, \ud x +  \ell.
\end{equation}
\end{pro} 
\begin{proof}
We first observe, in view of  \eqref{RSC1} and \eqref{RB1}, that $w^{\ell,\varepsilon}$ is governed by the equation 
\begin{equation}\label{weq}
\partial_t w^{\ell,\varepsilon} + \partial_x \left(b^{\ell,\varepsilon} w^{\ell,\varepsilon} \right) + \ell^2 \partial_x P^{\ell,\varepsilon} = \varepsilon \partial_x^2 w^{\ell,\varepsilon}, \qquad  w^{\ell,\varepsilon}|_{t=0}=w^{\ell,\varepsilon}_0 \bydef u^{\ell,\varepsilon}_0 - u_0^\varepsilon,
\end{equation}
where we have computed that  
\begin{align*}
f(u^{\ell,\varepsilon}) - f(u^{\varepsilon}) &= \int_0^1 \frac{\ud\ \,  }{\ud r} f \left( r  u^{\ell,\varepsilon}  +  (1-r)  u^{\varepsilon} \right) \ud r\\
&= \left( \int_0^1 f '\left( r  u^{\ell,\varepsilon}  +  (1-r)  u^{\varepsilon} \right) \ud r \right) w^{\ell,\varepsilon}   \bydef  b^{\ell,\varepsilon} w^{\ell,\varepsilon} .
\end{align*}
  Now, noticing that $w^{\ell, \varepsilon}_0 \in L^1(\R)$ by virtue of  the interpolation inequality \eqref{interpolation:L1BV}, one can show by an energy method (similar to the proof of Lemma \ref{lemTV}, for instance), for fixed values of $\ell,\varepsilon \in (0,1]$,  that  
\begin{equation*}
w^{\ell, \varepsilon}  \in L^\infty_{\textrm{loc}}(\mathbb{R}^+; L^1(\R)) ,
\end{equation*}
for any $t\geqslant 0$, where, at this stage, the preceding bound is not necessarily   uniform when $\ell \to 0$.
 Therefore, by  Lebesgue differentiation theorem, the anti-derivative of $w^{\ell,\varepsilon}$, that is $ \zeta  ^{\ell, \varepsilon}$ which is introduced above, is well defined. More precisely, it satisfies that
 \begin{equation}\label{dxW:identity}
 	\partial_x  \zeta  ^{\ell,\varepsilon}(t,x)= w^{\ell,\varepsilon}(t,x), \quad \text{for all } (t,x)\in \mathbb{R}^+\times \mathbb{R},
 \end{equation}
 and  enjoys the bounds
 \begin{equation*}
 	 \zeta  ^{\ell,\varepsilon} \in L^\infty_{\mathrm{loc}} (\mathbb{R}^+; C_b^0(\mathbb{R}) ),
 \end{equation*}
 for any fixed $\ell,\varepsilon\in (0,1]$. 
  Additionally,  thanks to the estimates on $u^{\ell,\varepsilon}$ and $u^{\varepsilon}$ that we previously established in Section \ref{section:UB}, the following bounds
 \begin{equation*}
 	 \partial_x 	 \zeta  ^{\ell,\varepsilon} \in L^\infty (\mathbb{R}^+; L^2 (\mathbb{R}) ) \cap L^\infty_{\textrm{loc}} (\mathbb{R}^+;  \dot{W}^{1,1}(\R)),
\end{equation*} 
 hold uniformly with respect to the parameters $\ell,\varepsilon\in (0,1]$. 

 Next, using the indentity $\zeta^{\ell, \varepsilon}(0,\cdot) = \zeta_0^\ell \ast \varphi_\varepsilon$, we obtain that
$$ \lim_{x\to \infty} \zeta^{\ell, \varepsilon}(0,x)=0.$$
Thus, integrating \eqref{weq} over $\mathbb{R}$ implies that $ \zeta^  {\ell, \varepsilon}$ vanishes at infinity for all time.   Moreover, integrating   \eqref{weq} over $(-\infty,x)$ we deduce that $ \zeta  ^{\ell,\varepsilon}$ is governed by the equation  
\begin{equation}\label{Weq}
\partial_t   \zeta  ^{\ell,\varepsilon} + b^{\ell,\varepsilon} \partial_x  \zeta  ^{\ell,\varepsilon}  + \ell^2  P^{\ell,\varepsilon} = \varepsilon \partial_x^2  \zeta  ^{\ell,\varepsilon}.
\end{equation}
Notice that this equation can be recast as 
\begin{equation*}
	 \zeta  ^{\ell,\varepsilon}(t,\cdot)=   \zeta  ^{\ell,\varepsilon}(0,\cdot) +  \int_0^t \left( - b^{\ell,\varepsilon} \partial_x  \zeta  ^{\ell,\varepsilon}  - \ell^2  P^{\ell,\varepsilon} + \varepsilon \partial_x^2  \zeta  ^{\ell,\varepsilon} \right) (\tau,\cdot) \ud \tau .
\end{equation*}
In view of the aforementioned bounds on $ \zeta  ^{\ell,\varepsilon}$ and the estimates that we established in the previous section, one can show that the right-hand side belongs to $L^\infty_{\textrm{loc}}(\mathbb{R}^+;L^1(\mathbb{R}))$, whereby we deduce that  
 \begin{equation*}
 	 \zeta  ^{\ell,\varepsilon} \in L^\infty _{\text{loc}}(\mathbb{R}^+; L^1(\mathbb{R}) ),
 \end{equation*}
 for any $\ell,\varepsilon\in (0,1]$.
Let now $\delta \in (0,1]$. Considering $S_\delta$ as is previously defined in \eqref{Sdef} and multiplying \eqref{Weq} by $S_\delta'( \zeta  ^{\ell,\varepsilon})$ yields that  
\begin{equation}\label{W2}
\partial_t S_\delta( \zeta  ^{\ell,\varepsilon} ) + \partial_x \left(  b^{\ell,\varepsilon} S_\delta( \zeta  ^{\ell,\varepsilon} ) \right)  +   \ell^2   P^{\ell,\varepsilon}  S_\delta'( \zeta  ^{\ell,\varepsilon} ) \leqslant  \varepsilon   \partial_x \left( S'_\delta( \zeta  ^{\ell,\varepsilon} ) \partial_x  \zeta  ^{\ell,\varepsilon} \right)  +   S_\delta( \zeta  ^{\ell,\varepsilon} ) \partial_x b^{\ell,\varepsilon}.
\end{equation}  
Next, observe that the Ole\u{\i}nik inequalities \eqref{Oleinikellep} and \eqref{Ol:vscl} with \eqref{assum:flux} entail  that
\begin{equation*}
\partial_x b^{\ell,\varepsilon} \leqslant a(t) \eqdef \frac{c_2}{\frac{c_1  }{2} t  +  \frac{1}{M}}.
\end{equation*}
Thus, 
multiplying \eqref{W2} by $\exp \{ -\int_0^t a(s) \ud s \}$ and   integrating over $\R$, and employing the simple observation  $|S_\delta'(\zeta  ^{\ell,\varepsilon})| \leqslant 1$, we find that  
\begin{equation*}
\frac{\ud\ }{\ud t} \left( \mathrm{e}^{-\int_0^t a(s)\, \ud s}  \int_\R S_\delta( \zeta  ^{\ell,\varepsilon}) (t,x)\, \ud x \right) \leqslant    \int_\R \ell^2  P^{\ell,\varepsilon} (t,x)\, \ud x.
\end{equation*}
Therefore, by an integration in time and using Proposition \ref{mu=0}, we arrive at the bound 
\begin{equation*}
 \int_\R S_\delta( \zeta  ^{\ell,\varepsilon}) (t,x)\, \ud x  \lesssim_{T,u_0,c_1,c_2}  \int_\R |\zeta  ^{\ell,\varepsilon} (0,x)|\, \ud x +  \ell \leqslant \int_\R |\zeta  ^{\ell}_0 (x )|\, \ud x +  \ell,
\end{equation*}
where the last inequality follows from the fact that $\zeta^{\ell, \varepsilon}(0,\cdot) = \zeta_0^\ell \ast \varphi_\varepsilon$.
Finally, taking $\delta \to 0$ and using the monotone convergence theorem, we obtain \eqref{wep}, thereby completing the proof of the proposition.
\end{proof}

\subsection{Proof of the main theorem}\label{proof:main:thm} We are now in position to prove the main result of this paper, that is Theorem \ref{main:thm}, which is a direct consequence of the following slightly stronger version.
 
\begin{thm} \label{main:thm:general}
Let   $u$   be the unique entropy solution of the scalar conservation law  \eqref{SCL} with a uniformly convex flux $f$ satisfying \eqref{assum:flux} and  an   initial datum $u_0 \in L^2(\mathbb{R}) \cap BV(\mathbb{R}) $ such that 
 \begin{equation*}
 M \eqdef \sup_{x,y \in \R,\ x \neq y} \frac{u_0(x)-u_0(y)}{x-y} < \infty .
 \end{equation*}
	Let $u^\ell$ be any solution of \eqref{rB}, given by Theorem \ref{thm:existence},    for an initial datum satisfying   
	\begin{equation*} 
\|u_0^\ell\|_{L^2(\R)} \leqslant \|u_0\|_{L^2(\R)}, \qquad \|\partial_x u_0^\ell\|_{L^1(\R)} \leqslant \|u_0\|_{BV(\R)}, \qquad  \sup_{x\in \R}\partial_x u_0^\ell(x) \leqslant M
\end{equation*}
and 
\begin{equation*} 
\ell \|\partial_x u_0^\ell\|_{L^2(\R)} \lesssim_{u_0} 1 ,
\end{equation*} 
	uniformly in $\ell \in (0,1]$.  
 Assume further that 
	\begin{equation*}
		x\mapsto \zeta_0^\ell(x) \eqdef \int_{-\infty}^{x } \big (u_0^\ell(y) - u_0(y)\big )\,\ud y  \in L^1(\mathbb{R}),
	\end{equation*}
	for any  $\ell \in (0,1]$.
   Then, it holds, for any $T>0$,  that \begin{equation}\label{uell-u2}
\norm{u^\ell - u}_{L^\infty ([0,T];L^p(\R))} \lesssim_{u_0,T,c_1,c_2} \left(  \norm {\zeta_0^\ell}_{L^1(\mathbb{R})} +  \ell \right) ^\frac{1}{2 p},
\end{equation}
for any $\ell \in (0,1]$ and $p \in [1, \infty)$.
\end{thm} 
\begin{rem}
Notice that Theorem \ref{main:thm} is recovered by simply taking $u^\ell_0 = u_0$, for all $\ell \in (0,1]$.
\end{rem}\begin{proof}
 The proof will be achieved in two steps:
\begin{enumerate}
	\item Convergence of the approximate solutions   in Lebesgue spaces, by an interpolation argument.
	\item Convergence of the exact solutions   in Lebesgue spaces, by local stability with respect to  $\varepsilon$.
\end{enumerate}  
\subsection*{Step 1}  We begin   with writing, in view of the identity \eqref{dxW:identity} and the interpolation inequality \eqref{interpolation:L1BV}  
\begin{equation*}
	\norm {w^{\ell, \varepsilon}(t)}_{L^1(\mathbb{R})}  \lesssim \norm{ \zeta^{\ell,\varepsilon}(t)}_{L^1(\mathbb{R})}^{\frac{1}{2}} \norm{ \partial_x w^{\ell,\varepsilon}(t)}_{L^1(\mathbb{R})}^{\frac{1}{2}},
\end{equation*}
for any $\ell \in (0,1]$ and $t\geqslant 0$. Therefore, applying    \eqref{TVep}, \eqref{TVep0} and  \eqref{wep}, we arrive at the bound 
\begin{equation*}
\norm {w^{\ell, \varepsilon}}_{L^\infty ([0,T]; L^1(\mathbb{R}))} \lesssim_{T,u_0,c_1,c_2}   \left( \int_\R |\zeta_0^{\ell}(x)|\, \ud x +  \ell\right)^\frac{1}{2}.
\end{equation*}
Moreover,   using H\"older inequality with \eqref{Linfep} and \eqref{Lp}, one can also deduce, for any $p\in [1,\infty)$, that 
\begin{equation*}
\norm {w^{\ell, \varepsilon}}_{L^\infty ([0,T]; L^p(\mathbb{R}))} \lesssim_{T,u_0,c_1,c_2}    \left( \int_\R |\zeta_0  ^{\ell} (x)|\, \ud x +  \ell\right)^\frac{1}{2p}.
\end{equation*}
\subsection*{Step 2}  From \eqref{vanishingviscosity} and \eqref{vanishingviscosity0} we deduce that, up to an extraction of  a subsequence, we have the convergence 
\begin{equation*}
w^{\ell, \varepsilon} \to u^\ell - u \quad \mathrm{in}\  C([0,T];L^p_{\text{loc}}(\R)),
\end{equation*} 
as $\varepsilon \to 0$. Therefore, it follows as a consequence of the convergence result from the preceding step that  
\begin{equation*}
\norm {u^\ell(t) - u(t)}_{L^p([-n,n])} \leqslant \liminf_{\varepsilon \to 0} \norm {w^{\ell, \varepsilon}}_{C([0,T];L^p([-n,n]))} \lesssim_{T,u_0,c_1,c_2}     \left( \int_\R |\zeta_0  ^{\ell} (x)|\, \ud x +  \ell\right)^\frac{1}{2p} ,
\end{equation*}
for any $n \in \mathbb{N}^\ast$ and all $t \in [0,T]$. In the end, letting $n \to \infty$ yields the validity of \eqref{uell-u2}, thereby concluding the proof of Theorem  \ref{main:thm}.
\end{proof}

\appendix 

\section{Functional spaces: Interpolation and Embeddings}\label{appendix} 

In this appendix, we agree that $d\geqslant 1$ denotes the dimension of the space variable. We shall collect some general results  which cover the overall functional embeddings that we routinely utilize in this paper.  This mainly involves properties of distributions belonging to Besov and $BV$ spaces, which we recall below. For clarity, let us point out here that one takeaway of this appendix is the justification of the known embedding 
\begin{equation*}
	\mathcal {M} (\mathbb{R}^d) \hookrightarrow \dot B ^{0}_{1,\infty} (\mathbb{R}^d),
\end{equation*}
where $\mathcal{M}(\mathbb{R}^d)$ stands for the space of Radon measures and $\dot B ^{0}_{1,\infty} (\mathbb{R}^d)$ is a homogeneous Besov space, which  implies that
\begin{equation*}
	BV (\mathbb{R}^d) \hookrightarrow  \dot B ^{1}_{1,\infty} (\mathbb{R}^d),
\end{equation*} where $BV$ denotes the space of locally-integrable functions with bounded variations.

\subsection{Besov and Sobolev spaces}
  We denote by $\mathcal{S}'(\mathbb{R}^d)$ the space of all tempered distributions defined on $\mathbb{R}^d$. Moreover, we say that a tempered distribution $f$ belongs to $\mathcal{S}'_h(\mathbb{R}^d)$ if it is not a polynomial near zero. More precisely, if it satisfies that \cite[Definition 1.26]{bcd11}
  \begin{equation*}
  	\lim_{\lambda \to \infty} \norm {\theta(\lambda D) f}_{L^\infty(\mathbb{R}^d)} = 0,
  \end{equation*}
for any $\theta \in \mathcal{D}(\mathbb{R}^d)$, where the symbol $\theta(D)$ denotes the Fourier multiplier by the smooth function $\theta$. Note that the preceding condition is automatically satisfied for any tempered distribution  whose Fourier transform is locally integrable near zero \cite[Example 1 page 22]{bcd11}. 

The homogeneous Besov space $\dot B^s_{p,q}(\mathbb{R}^d)$, for $s\in \mathbb{R}$ and $p, q \in [1,\infty]$,   is defined as the set of all tempered distributions $f$ in $\mathcal{S}'_h(\mathbb{R}^d)$ such that 
\begin{equation*}
	\norm {f}_{\dot B^s_{p,q}(\mathbb{R}^d)} \bydef \left( \sum_{j\in \mathbb{Z}} \left( 2^{js} \| \dot \Delta_j f\|_{L^p(\mathbb{R}^d)} \right)^q \right)^\frac{1}{q}<\infty,
\end{equation*} 
with the standard change of definition in the case $q=\infty$, where $(\dot \Delta_j)_{j\in \mathbb{Z}}$ denotes the usual (homogeneous) dyadic partition of unity, which is made of a family of a rescaled   smooth function supported away from zero. See \cite[Section 2.2]{bcd11} for the precise definition and important properties. We also recall  the identification (in terms of the semi-norms)
\begin{equation*}
	\dot B^s_{2,2}(\mathbb{R}^2) \approx \dot H^s(\mathbb{R}^d), \quad \text{for all} \quad s\in \mathbb{R},
\end{equation*}
which   defines the homogeneous Sobolev space $\dot H^s(\mathbb{R}^d)$ as a particular case of Besov spaces. 
Finally, we   define the Sobolev space $\dot W^ {s,p}(\mathbb{R}^d)$, for $p\in (1,\infty)$ and $s\in (0,\frac{d}{p})$, as the set of tempered distributions $f\in S'(\mathbb{R}^d) $ such that  
\begin{equation*}
	\norm {f}_{\dot W^{s,p}(\mathbb{R}^d)} \bydef \norm {(-\Delta)^{\frac{s}{2}} f}_{L^p(\mathbb{R}^d)}<\infty .
\end{equation*} 

The inhomogeneous Besov space $ B^s_{p,q}(\mathbb{R}^d)$, on the other hand,  is defined in a similar manner and it consists of all tempered distributions $f$ in $\mathcal{S}'(\mathbb{R}^d)$ such that 
\begin{equation*}
	\norm {f}_{ B^s_{p,q}(\mathbb{R}^d)} \bydef \left( \sum_{j\in \mathbb{Z}} \left( 2^{js} \|  \Delta_j f\|_{L^p(\mathbb{R}^d)} \right)^q \right)^\frac{1}{q}<\infty,
\end{equation*} 
with the standard change of definition in the case $q=\infty$, where $( \Delta_j)_{j\in \mathbb{Z}}$ denotes the usual (inhomogeneous) dyadic partition of unity. See \cite[Section 2.2]{bcd11}, again. Finally, we conclude by pointing out that 
\begin{equation*}
	B^s_{p,q}(\mathbb{R}^d) \approx  \dot  B^s_{p,q}(\mathbb{R}^d) \cap L^p(\mathbb{R}^d), 
\end{equation*}
for all $s>0$ and $p,q\in [1,\infty]$.

\subsection{Radon measures and $BV$ spaces}
The set of Radon measures is defined as the dual   space of continuous functions. More precisely, we introduce 
\begin{equation*}
	\mathcal{M}(\mathbb{R}^d) \bydef \left\{ f\in \mathcal{S}' (\mathbb{R}^d): \norm {f}_{\mathcal{M}}< \infty \right\},
\end{equation*}
where 
\begin{equation*}
	\norm {f}_{\mathcal{M}} \bydef \sup_{\overset{\varphi \in C^0(\mathbb{R}^d)}{\norm {\varphi}_{L^\infty }}\leqslant 1} |\langle f,\varphi\rangle |.
\end{equation*}
Moreover, we define the space of functions with bounded variations as 
\begin{equation*}
	BV(\mathbb{R}^d) \bydef \left\{ f\in L^1_{\text{loc}} (\mathbb{R}^d) \quad \text{with} \quad  \nabla f \in \mathcal{M}(\mathbb{R}^d) \right\},
\end{equation*}
equipped with the semi-norm 
\begin{equation*}
	\norm{f}_{BV(\mathbb{R}^d)} \bydef \norm f_{\mathcal {M}(\mathbb{R}^d)}.
\end{equation*}

\subsection{Embeddings}
 Here, we recall some functional inequalities which play a crucial role in our work. We begin with the classical (continuous) Sobolev embeddings, recast in the general context of Besov spaces \cite[Proposition 2.20]{bcd11}
 \begin{equation*} 
 	\dot B^{s}_{p,1} (\mathbb{R}^d) \hookrightarrow \dot B^{s}_{p,q} (\mathbb{R}^d) \hookrightarrow \dot B^{s-d\left (\frac{1}{p}-\frac{1}{r}\right)}_{r,q}(\mathbb{R}^d) \hookrightarrow \dot B^{s-d\left (\frac{1}{p}-\frac{1}{r}\right)}_{r,\infty}(\mathbb{R}^d),
 \end{equation*}
for all $s\in \mathbb{R}$, $1\leqslant p\leqslant r \leqslant \infty$ and any $q\in [1,\infty]$.  Note, moreover, that \begin{equation}\label{Besov:Lebesgue}
	\dot B^{0}_{p,1} (\mathbb{R}^d) \hookrightarrow L^p(\mathbb{R}^d) \hookrightarrow \dot B^{0}_{p,\infty} (\mathbb{R}^d),
\end{equation}
for any $p\in [1,\infty]$.
Another important feature of embeddings of Besov spaces is the gain in terms of the third  index in interpolation inequalities. More precisely, it holds, for any $f\in \dot B^{s_0}_{p,\infty } (\mathbb{R}^d)  \cap \dot B^{s_1}_{p,\infty } (\mathbb{R}^d)  $, that 
\begin{equation}\label{interpolation:1}
	\norm {f}_{\dot B^{s}_{p,1} (\mathbb{R}^d)  } \lesssim  \norm {f}_{\dot B^{s_0}_{p,\infty } (\mathbb{R}^d)  } ^{\theta} \norm {f}_{\dot B^{s_1}_{p,\infty } (\mathbb{R}^d)  }^{1-\theta},
\end{equation}
for any $p\in [1,\infty]$ and real parameters $s_0< s< s_1 $ with $s=\theta s_0 + (1-\theta)s_1$ and $\theta \in (0,1)$.

At last, the bridge between $BV$ and Besov  spaces can be apparent in the following embedding 
\begin{equation*} 
	\norm {f}_{\dot B^1_{1,\infty}(\mathbb{R}^d)} \lesssim \norm {f}_{BV(\mathbb{R}^d)},
\end{equation*}
for all $f\in BV(\mathbb{R}^d)$. This is a direct consequence of the bound \cite[Proposition 2.39]{bcd11}
\begin{equation}\label{Besov::BV}
	\norm {\nabla f}_{\dot B^0_{1,\infty}(\mathbb{R}^d)} \lesssim \norm {\nabla f}_{\mathcal{M}(\mathbb{R}^d)}.
\end{equation}
The justification  of this inequality  can be done by duality, writing \cite[Proposition 2.29]{bcd11}
\begin{equation*}
	\norm {\nabla f }_{\dot B^0_{1,\infty}(\mathbb{R}^d)}  \lesssim \sup_{\overset{\varphi \in \mathcal {S}(\mathbb{R}^d)}{\norm {\varphi}_{\dot {B}^0_{\infty,1}} }\leqslant 1}| \langle \nabla  f, \varphi\rangle|,
\end{equation*}
which, in view of the embedding 
\begin{equation*}
	 \dot {B}^0_{\infty,1}(\mathbb{R}^d) \hookrightarrow C^0 (\mathbb{R}^d) ,
\end{equation*}
yields the desired estimate 
\begin{equation*}
	\norm {\nabla f }_{\dot B^0_{1,\infty}(\mathbb{R}^d)}  \lesssim \sup_{\overset{\varphi \in \mathcal {C}^0(\mathbb{R}^d)}{\norm {\varphi}_{L^\infty }}\leqslant 1}| \langle \nabla  f, \varphi\rangle| = \norm {\nabla f}_{\mathcal{M}(\mathbb{R}^d)}.
\end{equation*}

   Finally, as a by-product of  the preceding interpolation and embedding inequalities, one can show the one dimensional control     \begin{equation}\label{interpolation:L1BV}
  	\norm {f}_{L^1(\mathbb{R})} \lesssim \norm {F}_{L^1(\mathbb{R})}^\frac{1}{2}  \norm {\nabla f}_{L^1 (\mathbb{R})}^\frac{1}{2},
  \end{equation}
  for any $f\in \dot W^{1,1}(\mathbb{R})$ with an anti-derivative $F$
  \begin{equation*}
  	x\mapsto F(x) \bydef \int_{-\infty}^x f(y)\, \ud y \in L^1(\mathbb{R}).
  \end{equation*} 
  This can be established by, first, writing in view of the definition of $F$ and \eqref{Besov:Lebesgue} that   \begin{equation*}
\begin{aligned}
	\norm {f}_{L^1(\mathbb{R})} 
	&\lesssim \norm{\nabla F}_{\dot B^{0}_{1,1}(\mathbb{R})}
	= \norm{F}_{\dot B^{1}_{1,1}(\mathbb{R})}.
	\end{aligned}
\end{equation*}
Then,  employing the interpolation inequality \eqref{interpolation:1} leads to the bound
\begin{equation*}
	\norm{ F}_{\dot B^{1}_{1,1}(\mathbb{R})} \lesssim \norm{ F}_{\dot B^{0}_{1,\infty}(\mathbb{R})}^{\frac{1}{2}} \norm{ F}_{\dot B^{2}_{1,\infty}(\mathbb{R})}^{\frac{1}{2}}.
\end{equation*} 
Therefore, by further appealing to the definition of $F$ and  \eqref{Besov::BV}, we infer that 
\begin{equation*}
	\norm {f}_{L^1(\mathbb{R})} 
	\lesssim \norm{ F}_{\dot B^{1}_{1,1}(\mathbb{R})} \lesssim \norm{ F}_{\mathcal{M}(\mathbb{R})}^{\frac{1}{2}} \norm{ \nabla f}_{\mathcal{M} (\mathbb{R})}^{\frac{1}{2}},
\end{equation*} 
whereby \eqref{interpolation:L1BV} follows.

\subsection*{Acknowledgement} 
The first author is supported by the LABEX MILYON (ANR-10-LABX-0070) of Universit\'e de Lyon, within the program ``Investissements d'Avenir'' (ANR-11-IDEX-0007) operated by the French National Research Agency (ANR). He is also supported by the Unit\'e de Math\'ematiques Pure et Appliqu\'ees, UMPA (CNRS and ENS de Lyon).

\end{document}